\documentclass[final]{siamltex}

\usepackage{amsmath,amssymb}
\usepackage{latexsym}
\usepackage{graphicx}
\usepackage{color}
\usepackage{subfigure}
\usepackage{hyperref}
\usepackage{multirow}

\def\norm#1{\left\| #1 \right\|}

\newcommand{\RR}{\mathbb{R}}
\newcommand{\ZZ}{\mathbb{Z}}

\newcommand{\EE}{\mathbb{E}}
\newcommand{\CC}{\mathbb{C}}
\newcommand{\SSS}{\mathbb{S}}

\newcommand{\BBB}{\mathcal{B}}
\newcommand{\III}{\mathcal{I}}
\newcommand{\OOO}{\mathcal{O}}
\newcommand{\LLL}{\mathcal{L}}

\newcommand{\GGG}{\mathcal{G}}

 \newtheorem{algorithm}[theorem]{Algorithm}

\newcommand{\vol}{\operatorname{vol}}
\newcommand{\cut}{\operatorname{cut}}

\newcommand{\argmin}{\operatorname{argmin}}

\newcommand{\Id}{\operatorname{I}}

\newcommand{\Ncut}{\operatorname{Ncut}}

\newcommand{\fruvone}{\vartheta}
\newcommand{\fruvonepar}{\vartheta^\ast}
\newcommand{\fruzone}{\zeta}

\newcommand{\fruztwo}{\eta}
\newcommand{\fruztwopar}{\eta^\ast}
\newcommand{\fruvtwo}{\nu}
\newcommand{\fruvtwopar}{\nu^\ast}
\newcommand{\IIIb}{\III b}
\newcommand{\RRdn}{\left(\RR^d\right)^n}
\newcommand{\RRReta}{\eta}

\newcommand{\proofb}[1]{{\noindent\emph{Proof.}} #1 \hfill$\Box$\linebreak}

\begin{document}

\title{A Cheeger Inequality for the Graph Connection Laplacian}



\author{
Afonso S. Bandeira\thanks{Program in Applied and Computational Mathematics (PACM),
Princeton University, Princeton, NJ 08544, USA ({\tt ajsb@math.princeton.edu}).}
\and
Amit Singer\thanks{PACM and Department of Mathematics,
Princeton University, Princeton, NJ 08544, USA ({\tt amits@math.princeton.edu}).}
\and
Daniel A. Spielman\thanks{Department of Applied Mathematics and Department of Computer Science, Yale University, PO Box 208285, New Haven, CT 06520-8285, USA ({\tt spielman@cs.yale.edu}).}
}

\maketitle

{\small
\begin{abstract}
The $O(d)$ Synchronization problem consists of estimating a set of $n$ unknown orthogonal $d\times d$ matrices $O_1,\ldots,O_n$ from noisy measurements of a subset of the pairwise ratios $O_iO_j^{-1}$. We formulate and prove a Cheeger-type inequality that relates a measure of how well it is possible to solve the $O(d)$ synchronization problem with the spectra of an operator, the graph Connection Laplacian. We also show how this inequality provides a worst case performance guarantee for a spectral method to solve this problem.
\end{abstract}
}


{\small
\begin{center}
\textbf{Keywords:}
$O(d)$ Synchronization, Cheeger Inequality, Graph Connection Laplacian, Vector Diffusion Maps.
\end{center}
}

\section{Introduction}

While the graph Laplacian is used to encode similarities between
  connected vertices,
  the graph Connection Laplacian endows the edges with transformations that describe the
  nature of the similarity.
For example, consider a collection of two-dimensional photos of a three-dimensional object that are taken from
  many angles as it spins in mid-air.
We could form a graph with one vertex for each photo by connecting
  photos that are similar.
By applying simple transformations, such as in-plane rotations, we will discover
  similarities between photos that are not apparent when we merely
  treat them as vectors of pixels.
In this case, we may also wish to keep track of the transformation under
  which a pair of photos is similar.
Singer and Wu~\cite{ASinger_HTWu_2011_VDM} defined the Connection Laplacian to
   encode this additional information.
The problem of using this information to assign a viewpoint
  to each picture is an instance of a synchronization problem on the graph.

More formally,
  the input to a synchronization problem over a group $\GGG$
  is an undirected graph  $G=(V,E)$ and
  a group element $\rho_{ij} \in \GGG$ for each edge $(i,j) \in E$, such that
$\rho_{ji} = \rho_{ij}^{-1}$.
We say that an assignment of group elements
  to vertices, also called a \textit{group potential},  $g:V\to \GGG$,
  \textit{satisfies} an edge $\rho_{ij}$ if
  $g_{i} = \rho_{ij} g_{j}$.
The objective in a synchronization problem is to find a
  group potential that satisfies the edges as much as possible.
When there is a $g$ that satisfies all of the edges, it is easy to find:
  one can arbitrarily fix the value of $g$ at one vertex and then
  iteratively set $g$ at neighbors of vertices whose
  values have already been set.
When there is no group potential that satisfies all of the edges,
  we must specify a measure of how well a group potential
  satisfies the edges. This is achieved by the notion of frustration
which will be defined in Section \ref{section:Frustration}.

We focus on the group $O(d)$  of $d\times d$ orthogonal matrices
  (rotations, reflections and compositions of both on $\RR^d$).
\textit{Please keep in mind that the ``O'' has nothing to do with
asymptotic notation.}
This group is of
particular interest in several applications. For example, when $d=1$,
i.e. $\GGG=O(1)\cong \ZZ/2\ZZ$, the solution to the synchronization
problem can be used to determine whether a manifold is
  orientable~\cite{ASinger_HTWu_2011_OrDM}.
The $\GGG=O (1)$ case is also a generalization of the \textsc{\textsc{Max-Cut}} problem \cite{MXGoemans_DPWilliamson_1995}:
  the group potential defines a partition of the vertices into two parts,
  and the group elements on edges specify whether the vertices they connect should be on the
  same or opposite sides of the partition.
In fact, our inequality for \textit{partial synchronization} can be
  understood as a generalization of Trevisan's
  inequality relating eigenvalues of the graph Laplacian and the
  maximum cut~\cite{LTrevisan_2008}.

Synchronization over $O(d)$
  plays a major role in an algorithm for the
  sensor network localization
  problem~\cite{MCucuringu_YLipman_ASinger_2011_sensor}.
The similar problem of synchronization over $SO(d)$,
  the group of rotations in $\RR^d$, also has several applications.
The problem over $SO (3)$ can be used
   for global alignment of 3-D scans~\cite{Tzeneva_SynchBunny}, and
  the problem over $SO(2)$ plays a major role in new algorithms for the
  cryo-electron microscopy problem (see
  \cite{ASinger_YShkolnisky_commonlines,ASinger_ZZhao_YShkolnisky_RHadani_cryo}).
Other applications of $SO(2)$ synchronization may be found
  in~\cite{ASinger_2011_angsync,Yu2012,Howard2010}.

Singer~\cite{ASinger_2011_angsync} proposed solving the $SO (2)$ synchronization
  problem by constructing a matrix, the Connection Laplacian, whose
  eigenvectors associated with the smallest eigenvalues would provide the group potential if it were possible
  to satisfy all the edges.
He then showed that under a model of random noise that prevents such a solution,
  a good solution can still be obtained by rounding the smallest eigenvectors.
A similar algorithm was proposed in~\cite{Cucuringu2012,ASinger_YShkolnisky_commonlines}
  for $SO(3)$.
These algorithms are analogous to spectral graph partitioning---the use of the smallest
  eigenvector of the Laplacian matrix to partition a graph~\cite{HagenKahng}.
The analysis of the $SO (2)$ algorithm under random noise can be viewed as an
  analog of McSherry's~\cite{McSherry} analysis of spectral partitioning.

Our paper contains three theorems.
We begin by considering the simpler problem of finding an assignment of unit vectors to each
  vertex that agrees with the transformations on the edges.
Motivated by Trevisan~\cite{LTrevisan_2008}, we first consider a variant of the problem
  in which we are allowed to find a \textit{partial} assignment in which we
  assign the zero vector to some of the vertices.
In Theorem~\ref{theorem:CI_Sd_CIpartial}, we prove
  a quadratic relationship between the smallest eigenvalue of
  the connection Laplacian and the minimum frustration of a partial assignment.
To analyze the case in which we must assign a unit vector to every vertex,
  we observe that partial assignments are really only required when the underlying
  graph has poor connectivity.
In Theorem~\ref{theorem:CI_Sd_CIfull}, we prove a relation between
  the minimum  frustration of a full assignment and the smallest
  eigenvalue of the Connection Laplacian \textit{and} the second-smallest eigenvalue
  of the underlying graph
  Laplacian.
We see that when the second-smallest eigenvalue of the underlying graph
 Laplacian is large (i.e., when it is a good expander), the
  minimum frustration is well approximated by the smallest eigenvalue of
  the Connection Laplacian.
Our main result is Theorem~\ref{theorem:CI_Od_MAIN}, in which we relate
  the minimum frustration of a group potential to the sum of the smallest
  $d$ eigenvalues of the Connection Laplacian \textit{and}
  the second-smallest  eigenvalue of the underlying graph Laplacian.

In the same way that the classical Cheeger's inequality provides a worst-case
 performance guarantee for spectral clustering, the three Theorems described above provide worst-case
 performance guarantees for a spectral method to solve each of the described Synchronization problems.
It is worth noting that, in a different setup, Trevisan~\cite{LTrevisan_2008} showed a particular case of
 this inequality, when the group is $O(1)\cong\ZZ/2\ZZ$ and all the offsets are $-1$. In that case the problem
 is equivalent to the \textsc{Max-Cut} problem. Therefore, this $O(1)$ inequality gives a performance guarantee for a spectral method to solve \textsc{Max-Cut}~\cite{LTrevisan_2008}.

It is worth mentioning that there are several other adaptations of Cheeger's inequality. Recent progress in multi-way partitioning problem gives a Cheeger inequality for the partitioning problem where one wants to partition the graph into more than $2$ subsets (see \cite{SOGharan_LTrevisan_2011,JRLee_SOGharan_LTrevisan_2011}). There is also a generalization of the Cheeger inequality for simplicial complexes, instead of graphs (see, \cite{Parzanchevski_Cheeger_2,Parzanchevski_Cheeger,Steenbergen_Cheeger}).

The rest of this section includes both mathematical preliminaries that are needed in later sections and the formulation of the problem. Section $2$ consists of our main contributions, algorithms to solve different forms of the Synchronization problem and Cheeger-type inequalities that provide guarantees for these methods, as well as a brief overview of the proofs. In Section $3$ we provide the rigorous proofs for the core results described in Section $2$. We discuss an $\ell_1$ version of the Synchronization problem in Section $4$ and show a few tightness results in Section $5$. We end with some concluding remarks and a few open problems in Section $6$.

\subsection{Notation}

Throughout the paper we use the notation $[n]$ to refer to
$\{1,\dots,n\}$. Also, we make use of several matrix and vector
notations. Given a matrix $A$ we denote by $\|A\|_F$ its Frobenius
norm. If $A$ is symmetric we denote by
$\lambda_1(A),\lambda_2(A),\dots$ its eigenvalues in increasing
order. Assuming further that $A$ is positive semi-definite we define
the $A$-inner product of vectors $x$ and $y$ as $\langle x,y\rangle_A
= x^TAy$ (and say that two vectors are $A$-orthogonal if this inner product is
zero). Also, we define the $A$-norm of $x$ as $\|x\|_A = \sqrt{\langle
x,x\rangle_A}$ and we represent the $\ell_2$ norm of $x$ by $\|x\|$. Given $x\in\RRdn$ we denote by $x_i$ the $i$-th
$d\times1$ block of $x$ (that will correspond to the value of $x$ on
the vertex~$i$) and, for any $u > 0$ we define $x^u$ as
\begin{equation}\label{def:xu}
x_i^u = \left\{ \begin{array}{ccl}  \frac{x_i}{\|x_i\|} & \text{ if } & \|x_i\|^2\geq u, \\ 0 & \text{ if } & \|x_i\|^2< u.  \end{array}  \right.
\end{equation}
Furthermore, for $u=0$ we denote $x^0$ by $\tilde{x}$, that is, 
\begin{equation}\label{def:tildex}
\tilde{x}_i = \left\{ \begin{array}{ccl}  \frac{x_i}{\|x_i\|} & \text{ if } & x_i \neq 0, \\ 0 & \text{ if } & x_i = 0.  \end{array}  \right.
\end{equation}
Finally, $\SSS^{d-1}$ denotes the unit sphere in $\mathbb{R}^d$.  

\subsection{Cheeger's Inequality and the Graph Laplacian}\label{section:ScalarCheeger}

Before considering the synchronization problem we will briefly present the classical graph Cheeger's inequality in the context of spectral partitioning. The material presented in this Section is well known but it will help motivate the ideas that follow in the next Sections.

Let $G=(V,E)$ be an undirected weighted graph with $n$ vertices. In this section we discuss the problem of partitioning the vertices in two similarly sized sets in a way that minimizes the cut: the volume of edges across the subsets (of the partition).

There are several ways to measure the performance of a particular partition of the graph, we will consider the one known as the Cheeger constant. Given a partition $(S,S^c)$ of $V$ let $h_S:= \frac{\cut(S)}{\min\{\vol(S),\vol(S^c)\}}$, where the value of the cut associated with $S$ is $cut(S) = \sum_{i\in S}\sum_{j\in S^c} w_{ij}$, its volume is $vol(S)= \sum_{i\in S}d_i$, and $d_i = \sum_{j\in V} w_{ij}$ is the weighted degree of vertex $i$. We want to partition the graph so that $h_S$ is minimized, and the minimum value is referred to as the Cheeger number of the graph, denoted $h_G = \min_{S\subset V}h_S$. Finding the optimal $S$ is known to be NP-hard, as it seems to require searching over an exponential number of possible partitions.

There is another way to measure the performance of a partition $(S,S^c)$ known as the normalized cut:
\[
\Ncut(S) =\cut(S)\left( \frac1{\vol(S)}+\frac1{\vol(S^c)} \right).
\]
As before, we want to find a subset with as small of an $\Ncut$ as possible. Note that the normalized cut and the Cheeger constant are closely related:
\[
\frac12 \Ncut(S) \leq  h_S \leq \Ncut(S).
\]
Let us introduce a few important definitions. Let $W_0$ be the weighted adjacency matrix of $G$ and $D_0$ the degree matrix, a diagonal matrix with elements $d_i$. If we consider a vector $f\in\RR^n$ whose entries take only $2$ possible values, one associated with vertices in $S$ and another in $S^c$, then the quadratic form $Q_f= \frac12\sum_{ij}w_{ij} \left(f_i-f_j\right)^2$ is of fundamental importance as a measure of the cut between the sets. The symmetric positive semi-definite matrix that corresponds to this quadratic form, $L_0$, is known as the graph Laplacian of $G$. It is defined as $L_0 = D_0 - W_0$ and satisfies $v^TL_0v=Q_v$ for any $v\in \mathbb{R}^n$. It is also useful to consider the normalized graph Laplacian $\LLL_0 = D_0^{-1/2}L_0D_0^{-1/2} = \Id - D_0^{-1/2}W_0D_0^{-1/2} $, which is also a symmetric positive semi-definite matrix.

Let us represent a partition $(S,S^c)$ by a cut-function $f_S:V\to\RR$ given by
\[
f_S(i) = \left\{ \begin{array}{rcl} \sqrt{\frac{\vol(S^c)}{\vol(S)\vol(G)}} & \text{ if }  & i\in S, \\  -\sqrt{\frac{\vol(S^c)}{\vol(S)\vol(G)}} & \text{ if } & i\in S^c. \end{array} \right.
\]

It is straightforward to show that $Q_{f_S} = f_S^TL_0f_S  = \Ncut(S)$, $f_S^TD_0f_S=1$, and $f_S^TD_0 {\bf 1}= 0$, where ${\bf 1}$ is the all-ones vector in $\mathbb{R}^n$. This is the motivation for a spectral method to approximate the minimum normalized cut problem. If we drop the constraint that $f$ needs to be a cut-function and simply enforce the properties established above then one would formulate the following relaxed problem
\begin{equation}\label{spectralclusteringev}
\min_{f:V\to\RR, f^TD_0f=1, f^TD_0 {\bf 1}=0} f^TL_0f.
\end{equation}
Since ${\bf 1}^TL_0{\bf 1}=0$, we know by the Courant-Fisher formula that (\ref{spectralclusteringev}) corresponds to an eigenvector problem whose minimum is $\lambda_2(\LLL_0)$ and whose minimizer can be obtained by the corresponding eigenvector.

Since problem (\ref{spectralclusteringev}) is a relaxation of the minimum $\Ncut$ problem we automatically have $\frac12\lambda_2(\LLL_0) \leq \frac12\min_{S\subset V}\Ncut\leq h_G$. Remarkably one can show that the relaxation is not far from the partitioning problem. In fact, one can round the solution of (\ref{spectralclusteringev}) so that it corresponds to a partition $(S,S^c)$ of $G$, whose $h_S$ we can control. This is made precise by the following classical result in spectral graph theory (several different proofs for this inequality can be found in \cite{Chung_CheegersIneq}):

\begin{theorem}[Cheeger Inequality \cite{NAlon_1986,NAlon_VMilman_1986}]\label{CheegerInequality56}
Let $G=(V,E)$ be a graph and $\LLL_0$ its normalized graph Laplacian. Then
\[
\frac12\lambda_2(\LLL_0)\leq h_G \leq \sqrt{2\lambda_2(\LLL_0)},
\]
where $h_G$ is the Cheeger constant of $G$.
Furthermore, the bound is constructive: using the solution of the eigenvector problem one can produce partition $(S,S^c)$ that achieves the upper bound $\sqrt{2\lambda_2(\LLL_0)}$.
\end{theorem}

An alternative way to interpret Theorem \ref{CheegerInequality56} is through random walks on graphs. Note that the matrix $D_0^{-1}W_0$ is the transition probability matrix of a random walk in $G$, whose transition probabilities are proportional to the edge weights. It is known that the eigenvalues of $\LLL_0$ encode important information about the random walk. In fact, the second smallest eigenvalue\footnote{We note that the smallest eigenvalue is always $0$.} is a good measure of how well the random walk mixes. More specifically, the smaller $\lambda_2(\LLL_0)$, the slower the convergence to the limiting stationary distribution. It is clear that clusters will constitute obstacles to rapid mixing of the random walk, since the probability mass might be trapped inside such a set for a while. Cheeger's inequality (Theorem \ref{CheegerInequality56}) shows that, in some sense, these sets are the only obstacles to rapid mixing.

\subsection{Frustration, Vector-Valued Walks and the Connection Laplacian}\label{section:Frustration}

If, in addition to a graph, we are given an orthogonal
  transformation $\rho_{ij}\in O(d)$ for each edge $(i,j) \in E$,
  we can consider a random walk that  takes the transformations
  into account.
One way of doing this is by defining a random walk that,
instead of moving point masses, moves a vector from vertex to vertex
and transforms it via the orthogonal transformation associated with
the edge.
One can similarly define a random walk that moves group elements
  on vertices.
The Connection Laplacian was defined by
  Singer and Wu~\cite{ASinger_HTWu_2011_VDM} to measure the convergence
  of such random walks.
The construction requires that $\rho_{ji} = \rho_{ij}^{-1} = \rho_{ij}^{T}$.
Define the symmetric matrix
  $W_1\in\RR^{dn\times dn}$ so that the $(i,j)$-th $d\times d$ block
  is given by $(W_1)_{ij} = w_{ij}\rho_{ij}$, where $w_{ij}$ is the
  weight of the edge $(i,j)$.
Also, let $D_1\in\RR^{dn\times dn}$, be the
diagonal matrix such that $(D_1)_{ii} = d_i I_{d\times d}$. We assume $d_i>0$, for every $i$. The graph
Connection Laplacian $L_1$ is defined to be $L_1 = D_1 - W_1$, and the
normalized graph Connection Laplacian is
\[
\LLL_1 = \Id - D_1^{-1/2}W_1D_1^{-1/2}.
\]

If $v:V\to\SSS^{d-1}$ assigns
  a unit vector in $\RR^{d}$ to each vertex,
  we may think of $v$ as a vector in $dn$ dimensions.
In this case the quadratic form
\[
  v^{T} L_{1} v =
  \sum_{(i,j)\in E}w_{ij}\left\| v_i - \rho_{ij}v_j\right\|^2 = \frac12\sum_{i,j}w_{ij}\left\| v_i - \rho_{ij}v_j\right\|^2
\]
is a measure of how well $v$ satisfies the edges.
This will be zero if $v_{i} = \rho_{ij} v_{j}$ for all edges
  $(i,j)$. As $w_{ij}=0$ when $(i,j)\notin E$, we can sum over all pairs of vertices  without loss of generality. An assignment satisfying all edges will correspond to a stationary distribution in the vector-valued random walk.

Following our analogy with Cheeger's inequality for the normalized
  graph Laplacian, we normalize this measure by defining the \textit{frustration}
  of $v$ as
\begin{equation}\label{frustrationofv}
  \fruztwo (v) =
\frac{ v^{T} L_{1} v}{v^{T} D_{1} v}
 = \frac12
\frac{\sum_{i,j}w_{ij}\left\| v_i - \rho_{ij}v_j\right\|^2
  }{
    \sum_{i} d_{i} \norm{v_{i}}^{2}
  }.
\end{equation}
We then define the $\SSS^{d-1}$ frustration
  constant of $G$ as
\begin{equation}\label{defofSdm1frustration}
\fruztwo_G = \min_{v:V\to\SSS^{d-1}}\RRReta(v).
\end{equation}
The smallest eigenvalue of $\LLL_{1}$ provides a relaxation
  of $\fruztwo_G$, as
\[
\lambda_1(\LLL_1) =  \min_{z\in\RR^{dn}}\frac{z^T \LLL_1 z}{z^Tz} = \min_{x\in\RR^{dn}}\frac{(D_1^{\frac12}x)^T \LLL_1 (D_1^{\frac12}x)}{(D_1^{\frac12}x)^T(D_1^{\frac12}x)} =  \min_{x\in\RR^{dn}}\frac{x^T L_1 x}{x^TD_1x} = \min_{x:V\to\RR^{d}} \RRReta(x).
\]

If there is a group potential $g : V \to O (d)$
  that satisfies all the edges (which would again correspond to a stationary distribution for the $O(d)$-valued random walk),
  then we can obtain $d$ orthogonal vectors
  on which the quadratic form defined by $\LLL_{1}$ is zero.
For each $1 \leq k \leq d$ we obtain one of these vectors by
  setting $v (i)$ to the $k$th column of $g (i)$ for all $i \in V$.
In particular, this means that the columns of the matrices of the group potential
  that satisfies all of the edges lie in the nullspace of $\LLL_{1}$. Since $g(i)\in O(d)$ these vectors are orthogonal.
If $G$ is connected, one can show that these are the only vectors in the
  nullspace of $\LLL_{1}$.
This observation is the motivation for the use of a spectral algorithm
  for synchronization.

We define the frustration of a group potential $g : V \to O (d)$ to be
\begin{equation}\label{def:Odfrustration}
\fruvtwo(g) = \frac1{2d}\frac1{\vol(G)}\sum_{i,j}w_{ij}\|g_i - \rho_{ij}g_j\|_F^2.
\end{equation}
We then define the $O(d)$ frustration constant of $G$ to be
\[
\fruvtwo_G = \min_{g:V\to O(d)}\fruvtwo(g).
\]
In Theorem~\ref{theorem:CI_Od_MAIN}, we prove that this frustration constant is small
  if and only if the sum of the first $d$ eigenvalues of $\LLL_{1}$ is small as well.

\section{Cheeger's type inequalities for the synchronization problem}


In this Section we present our main results. We present three spectral algorithms to solve three different formulations of synchronization problems and obtain for each a guarantee of performance in the form of a Cheeger's type inequality. We will briefly summarize both the results and the ideas to obtain them, leaving the rigorous proofs to Section \ref{section:proofs}.

We start by considering the $\SSS^{d-1}$ synchronization problem. This corresponds to finding, for each vertex $i$ of the graph, a vector $v_i\in\SSS^{d-1}$ in way that for each edge $(i,j)$ the vectors agree with the edges, meaning $v_i= \rho_{ij}v_j$. Since this might not always be possible we look for a function $v:V\to\SSS^{d-1}$ for which the frustration $\RRReta(v)$ is minimum (see (\ref{frustrationofv})). Motivated by an algorithm to solve \textsc{Max-Cut} by Trevisan~\cite{LTrevisan_2008}, we first consider a version of the problem for which we allows ourselves to synchronize only a subset of the vertices, corresponding to the partial synchronization in $\SSS^{d-1}$. We then move on to consider the full synchronization problem in $\SSS^{d-1}$.

Finally we will present our main result, an algorithm for $O(d)$ synchronization and a Cheeger-like inequality that equips it with a worst-case guarantee. Recall that the $O(d)$ synchronization corresponds to finding an assignment of an element $g_i\in O(d)$ to each vertex $i$ in a way that minimizes the discrepancy with the pairwise measurements $\rho_{ij}\sim g_ig_j^{-1}$ obtained for each edge. This corresponds to minimizing the $O(d)$ frustration, $\fruvtwo(g)$, (see (\ref{def:Odfrustration})).

\subsection{Partial synchronization in $\SSS^{d-1}$}

The motivation for considering a spectral relaxation for the synchronization problem in $\SSS^{d-1}$ is the observation that $\lambda_1(\LLL_1) =  \min_{x:V\to\RR^{d}} \RRReta(x)$. In order to understand how tight the relaxation is we need to relate $\lambda_1(\LLL_1)$ with $\fruztwo_G = \min_{x:V\to\SSS^{d-1}} \RRReta(x)$.

Consider, however, the following example:  a graph consisting of two disjoint components, one whose $\rho_{ij}$ measurements are perfectly compatible and another one on which they are not. Its graph Connection Laplacian would have a non-zero vector in its null space, corresponding to synchronizations  on the compatible component and zero on the incompatible part (thus $\lambda_1(\LLL_1) = 0$). On the other hand, the constraint that $v$ has to take values on $\SSS^{d-1}$, will force it to try to synchronize the incompatible part thereby bounding $\fruztwo_G$ away from zero. This example motivates a different formulation of the $\SSS^{d-1}$ synchronization problem where vertices are allowed not to be labeled (labeled with $0$). We thus define the partial $\SSS^{d-1}$ frustration constant of $G$, as the minimum possible frustration value for such an assignment,
\begin{equation}\label{form:def:zeta_G}
\fruztwopar_G = \min_{v:V\to\SSS^{d-1}\cup\{0\}}\fruztwo(v).
\end{equation}

We propose the following algorithm to solve the partial $\SSS^{d-1}$ synchronization problem.

\begin{algorithm}\label{Alg:PartialSdSynch}
Given a graph $G = (V,E)$ and a function $\rho: E\to O(d)$, construct the normalized Connection Laplacian $\LLL_1$ and the degree matrix $D_1$. Compute $z$, the eigenvector corresponding to the smallest eigenvalue of $\LLL_1$. Let $x=D_1^{-\frac12}z$. For each vertex index $i$, let $u_i = \|x_i\|$, and set $v^i:V\to\SSS^{d-1}\cup\{0\}$ as $v^i = x^{u_i}$, according to (\ref{def:xu}).
Output $v$ equal to the $v^i$ that minimizes $\fruztwo\left(v^i \right)$.
\end{algorithm}

Lemma \ref{lemma:CI_Sd_fromTrevisan_PARTIAL} guarantees that the solution $v$ given by Algorithm \ref{Alg:PartialSdSynch} satisfies $\fruztwo(v)\leq \sqrt{10\eta(x)}$. Since $x$ was computed so that $\fruztwo(x) = \lambda_1(\LLL_1)$, Algorithm \ref{Alg:PartialSdSynch} is guaranteed to output a solution $v$ such that
\[
\fruztwo(v)\leq \sqrt{10\lambda_1(\LLL_1)}.
\]
Note that $ \lambda_1(\LLL_1) \leq \fruztwopar_G$, which is the optimum value for the  partial $\SSS^{d-1}$ synchronization problem (see (\ref{form:def:zeta_G})).
The proof for Lemma \ref{lemma:CI_Sd_fromTrevisan_PARTIAL} will appear below. The idea to show that the rounding, from $x$, to the solution $v$ done by Algorithm \ref{Alg:PartialSdSynch} produces a solution with $\fruztwo(v)\leq \sqrt{10\eta(x)}$ is to use the probabilistic method. One considers a random rounding scheme by rounding $x$ as in Algorithm \ref{Alg:PartialSdSynch} and (\ref{def:xu}) but thresholding at a random value $u$, drawn from a well-chosen distribution. One then shows that, in expectation, the frustration of the rounded vector is bounded by $\sqrt{10\fruztwo(x)}$. This automatically ensures that there must exist a value $u$ that produces a solution with frustration bounded by $\sqrt{10\fruztwo(x)}$. The rounding described in Algorithm \ref{Alg:PartialSdSynch} runs through all possible such roundings and is thus guaranteed to produce a solution satisfying the bound. An $O(1)$ version of this algorithm and analysis appeared in~\cite{LTrevisan_2008}, when $\rho$ is the constant function equal to $-1$, in the context of the \textsc{Max-Cut} problem. In fact, if $d=1$ the factor $10$ can be substituted by $8$ and the stronger inequality holds $\fruztwo(v)\leq \sqrt{8\lambda_1(\LLL_1)}$.

The above performance guarantee for Algorithm \ref{Alg:PartialSdSynch} automatically implies the following Cheeger-like inequality.

\begin{theorem}\label{theorem:CI_Sd_CIpartial}
Let $G = (V,E)$ be a weighted graph. Given a function $\rho: E\to O(d)$, let $\fruztwopar_G$ be the partial $\SSS^{d-1}$ frustration constant of $G$ and $\lambda_1(\LLL_1)$ the smallest eigenvalue of the normalized graph Connection Laplacian. Then
\begin{equation}\label{form:theorem:CI_Sd_CIpartial}
\lambda_1(\LLL_1) \leq \fruztwopar_G \leq \sqrt{10\lambda_1(\LLL_1)}.
\end{equation}
Furthermore, if $d=1$, the stronger inequality holds, $\fruztwopar_G \leq \sqrt{8\lambda_1(\LLL_1)}$.
\end{theorem}

We note that Trevisan~\cite{LTrevisan_2008}, in the context of \textsc{Max-Cut}, iteratively performs this partial synchronization procedure in the subgraph composed of the vertices left unlabeled by the previous iteration, in order to label the entire graph. We, however, consider only one iteration.

\subsection{Full synchronization in $\SSS^{d-1}$}

In this section we adapt Algorithm \ref{Alg:PartialSdSynch} to solve (full) synchronization in $\SSS^{d-1}$ and  show performance guarantees, under reasonable conditions, by obtaining bounds for $\fruztwo_G$, the frustration constant for synchronization in $\SSS^{d-1}$. The intuition given to justify the relaxation to partial $\SSS^{d-1}$ synchronization was based on the possibility of poor connectivity of the graph (small spectral gap). In this section we show that poor connectivity, as measured by a small spectral gap in the normalized graph Laplacian, is the only condition under which one can have large discrepancy between the frustration constants and the spectra of the graph Connection Laplacian. We will show that, as long as the spectral gap is bounded away from zero, one can in fact control the full frustration constants. 

\begin{algorithm}\label{Alg:FullSdSynch}
Given a weighted graph $G = (V,E)$ and a function $\rho: E\to O(d)$, construct the normalized Connection Laplacian $\LLL_1$ and the degree matrix $D_1$. Compute $z$, the eigenvector corresponding to the smallest eigenvalue of $\LLL_1$. Let $x=D_1^{-\frac12}z$. Output the solution $v:V\to\SSS^{d-1}\cup\{0\}$ where each $v_i$ is defined as
\[
v_i = \frac{x_i}{\|x_i\|}.
\]
If $x_i=0$, have $v_i$ be any vector in $\SSS^{d-1}$.
\end{algorithm}

Similarly to Algorithm \ref{Alg:PartialSdSynch}, Lemma \ref{lemma:CI_Sd_CIfull} guarantees that the solution $v$ given by Algorithm \ref{Alg:FullSdSynch} satisfies $\fruztwo(v)\leq 44\frac1{\lambda_2(\LLL_0)}\fruztwo(x)$. Again, since $x$ was computed so that $\fruztwo(x) = \lambda_1(\LLL_1)$, then Algorithm \ref{Alg:PartialSdSynch} is guaranteed to output a solution $v$ such that
\[
\fruztwo(v)\leq 44\frac{\lambda_1(\LLL_1)}{\lambda_2(\LLL_0)}.
\]
Recall that, trivially, $ \lambda_1(\LLL_1) \leq \fruztwo_G$, which is the optimum value for the (full) $\SSS^{d-1}$ synchronization problem (see (\ref{defofSdm1frustration})).
The proof for Lemma \ref{lemma:CI_Sd_CIfull} is also deferred until Section \ref{section:proofs}. The idea here is to look at the vector of the local norms of $x$: $n_x\in\RR^n$ where $n_x(i) = \|x_i\|$. It is not hard to show that $\frac{n_x^TL_0n_x}{n_x^TD_0n_x} \leq \fruztwo(x)$, which means that, if $\fruztwo(x)$ is small then $n_x$ cannot vary much between two vertices that share an edge. Since $\lambda_2(\LLL_0)$ is large one can show that such a vector needs to be close to constant, which means that the norms of $x$ across the vertices are similar. If the norms were all the same then the rounding $v_i = \frac{x_i}{\|v_i\|}$ would not affect the value of $\fruztwo(\cdot)$, we take this slightly further by showing that if the norms are similar then we can control how much the rounding affects the penalty function.

The above performance guarantee for Algorithm \ref{Alg:FullSdSynch} automatically implies another Cheeger-like inequality.

\begin{theorem}\label{theorem:CI_Sd_CIfull}
Let $G = (V,E)$ be a graph. Given a function $\rho: E\to O(d)$, let $\fruztwo_G$ be the $\SSS^{d-1}$ frustration constants of $G$, $\lambda_1(\LLL_1)$ the smallest eigenvalue of the normalized graph Connection Laplacian and $\lambda_2(\LLL_0)$ the second smallest eigenvalue of the normalized graph Laplacian. Then,
\[
\lambda_1(\LLL_1) \leq \fruztwo_G \leq 44\frac{\lambda_1(\LLL_1)}{\lambda_2(\LLL_0)}.
\]
\end{theorem}

\subsection{The $O(d)$ synchronization problem}

We present now our main contribution, a spectral algorithm for $O(d)$ synchronization together with a Cheeger-type inequality that provides a worst-case performance guarantee for the algorithm.

Before presenting the Algorithm and the results let us note the differences between this problem and the $\SSS^{d-1}$ synchronization problem, presented above. For the $\SSS^{d-1}$ case, the main difficulty that we faced in trying to obtain candidate solutions from eigenvectors was the local unit norm constraint. This is due to the fact that the synchronization problem requires its solution to be a function from $V$ to $\SSS^{d-1}$, corresponding to a vector in $\RR^{dn}$ whose vertex subvectors have unit norm, while the eigenvector, in general, does not satisfy such a constraint. Nevertheless, the results in the previous section show that, by simply rounding the eigenvector, one does not lose more than a linear term, given that the graph Laplacian has a spectral gap bounded away from zero.

However, the $O(d)$ synchronization setting is more involved. The reason being that, besides the local normalization constraint, there is also a local orthogonality constraint (at each vertex, the $d$ vectors have to be orthogonal so that they can be the columns of an orthogonal matrix). For $\SSS^{d-1}$ we locally normalized the vectors, by choosing for each vertex the unit vector closest to $x_i$. For $O(d)$ synchronization we will pick, for each vertex, the orthogonal matrix closest (in the Frobenius norm) to the matrix $\left[x_i^1\cdots x_i^d\right]$, where $x_i^j$ corresponds to the $d-$dimensional vector assigned to vertex $i$ by the $j$'th eigenvector. This rounding can be achieved by the Polar decomposition. Given a $d\times d$ matrix $X$, the matrix $U(X)$, solution of $\min_{U\in O(d)}\|U(X)-X\|_F$, is one of the components of the Polar decomposition of $X$ (see \cite{Higham_PolarDecomposition,Li_StabilityPolarDecomposition} and references therein). We note that $U(X)$ can be computed efficiently through the SVD decomposition of $X$. In fact, given the SVD of $X$, $X = U\Sigma V^T$, the closest orthogonal matrix to $X$ is given by $U(X) = UV^T$ (see \cite{Higham_PolarDecomposition}). This approach is made precise in the following spectral algorithm for $O(d)$-synchronization.

\begin{algorithm}\label{Alg:OdSync}
Given a weighted graph $G = (V,E)$ and a function $\rho: E\to O(d)$, construct the normalized Connection Laplacian $\LLL_1$ and the degree matrix $D_1$. Compute $z^1,\dots,z^d$, the first $d$ eigenvectors corresponding to the $d$ smallest eigenvalues of $\LLL_1$. Let $x^j=D_1^{-\frac12}z^j$, for each $j=1,\dots,d$. Output the solution $g:V\to O(d)$ where each $g_i$ is defined as
\[
g_i = U(X_i),
\]
where $X_i = \left[x_i^1\cdots x_i^d\right]$ and $U(X_i)$ is the closest orthogonal matrix of $X_i$, which can be computed via the SVD of $X_i$, if $X_i = U_i\Sigma_i V_i^T$, then $U(X_i) = U_iV_i^T$.
If $X_i$ is singular\footnote{In this case the uniqueness of $U(X_i)$ is not guaranteed and thus the map is not well-defined.} simply set $U(X_i)$ to be $\Id_d$.
\end{algorithm}

Similarly to how the performance of the $\SSS^{d-1}$ synchronization algorithms was obtained, Lemma \ref{lemma:CI_Od_MAINLEMMA} bounds the effect of the rounding step in Algorithm \ref{Alg:OdSync}. Before rounding, the frustration of the solution $\left[x^1\cdots x^d\right]$ is $\frac1d\sum_{i=1}^d\RRReta\left(x^i\right)$. Lemma \ref{lemma:CI_Od_MAINLEMMA} guarantees that the solution $g$ obtained by the rounding in Algorithm \ref{Alg:OdSync} satisfies
$\fruvtwo(g) \leq 1026 d^3 \frac1{\lambda_2(\LLL_0)}\sum_{i=1}^d\RRReta\left(x^i\right)$.
Because of how the vectors $x^1,\dots,x^d$ were built, $\sum_{i=1}^d\RRReta\left(x^i\right)= \sum_{i=1}^d\lambda_i(\LLL_1)$, and this means that the solution $g$ computed by Algorithm \ref{Alg:OdSync} satisfies
\[
\fruvtwo(g) \leq 1026d^3 \frac1{\lambda_2(\LLL_0)} \sum_{i=1}^d\lambda_i(\LLL_1).
\]

This performance guarantee automatically implies our main result, a Cheeger inequality for the Connection Laplacian.

\begin{theorem}\label{theorem:CI_Od_MAIN}
Let $\lambda_i(\LLL_1)$ and $\lambda_i(\LLL_0)$ denote the $i$-th smallest eigenvalues of the normalized Connection Laplacian $\LLL_1$ and the normalized graph Laplacian $\LLL_0$ respectively. Let $\fruvtwo_G$ denote the frustration constant for $O(d)$ Synchronization.
Then,
\[
\frac{1}{d} \sum_{i=1}^d\lambda_i(\LLL_1) \leq \fruvtwo_G \leq 1026d^3 \frac1{\lambda_2(\LLL_0)} \sum_{i=1}^d\lambda_i(\LLL_1).
\]
\end{theorem}

Note that, once again, the lower bound is trivially obtained by noting that the eigenvector problem is a relaxation of the original synchronization problem.

Although the rigorous statement and proof of Lemma \ref{lemma:CI_Od_MAINLEMMA} will be presented in Section \ref{section:proofs}, we give a brief intuitive explanation of how the result is obtained.

As discussed above, the performance guarantee for Algorithm \ref{Alg:FullSdSynch} relies on a proper understanding of the effect of the rounding step. In particular we showed that if $\lambda_2(\LLL_0)$ is small, then locally normalizing the candidate solution (which corresponds to the rounding step) has an effect over the penalty function that we can control. The case of $O(d)$ Synchronization is dealt with similarly.
Instead of local normalization, the rounding step for Algorithm \ref{Alg:OdSync} is based on the polar decomposition. We start by understanding when the polar decomposition is stable (in the sense of changing the penalty function on a given edge) and see that this is the case when the candidate solution $X_i\in\RR^{d\times d}$ is not close to being singular. The idea then is to show that only a small portion (which will depend on $\sum_{i=1}^d\lambda_i(\LLL_1)$ and $\lambda_2(\LLL_0)$) of the graph can have candidate solutions $X_i$ close to singular and use that to show that, overall, we can bound the harmful contribution potentially caused by the rounding procedure on the penalty function.

\section{Proof of the main results}\label{section:proofs}

In this Section we prove the results described above.

\subsection{Proofs for Synchronization in $\SSS^{d-1}$}
We start with the main Lemma regarding partial $\SSS^{d-1}$ Synchronization.

\begin{lemma}\label{lemma:CI_Sd_fromTrevisan_PARTIAL}
Given $x\in\RR^{dn}$ there exists $u>0$ such that
\[
\fruztwo(x^u) \leq \sqrt{10\RRReta(x)}.
\]
Moreover, if $d=1$ the right-hand side can be replaced by $\sqrt{8\RRReta(x)}$.
\end{lemma}

\begin{proof}
This Lemma immediately follows from Lemma \ref{lemmaalsogoodforl1} as
\[
\fruztwo(x^u) = \frac12\frac{\sum_{ij}w_{ij}\|x^u_i-\rho_{ij}x^u_j\|^2}{\sum_id_i\|x_i^u\|^2} \leq \left( \frac12\max_{i,j}\| x^u_i-\rho_{ij}x^u_j \| \right)\frac{\sum_{ij}w_{ij}\|x^u_i-\rho_{ij}x^u_j\|}{\sum_id_i\|x_i^u\|} \leq \frac{\sum_{ij}w_{ij}\|x^u_i-\rho_{ij}x^u_j\|}{\sum_id_i\|x_i^u\|},
\]
where the last inequality was obtained by noting that $\| x^u_i-\rho_{ij}x^u_j \| \leq \| x^u_i\|+\|x^u_j\| \leq 2$
\end{proof}

\begin{lemma}\label{lemmaalsogoodforl1}
Given $x\in\RR^{dn}$ there exists $u>0$ such that
\[
 \frac{\sum_{ij}w_{ij}\|x^u_i-\rho_{ij}x^u_j\|}{\sum_id_i\|x^u_i\|} \leq \sqrt{10\RRReta(x)}.
\]
Moreover, if $d=1$ the right-hand side can be replaced by $\sqrt{8\RRReta(x)}$.
\end{lemma}

\begin{proof}
Let us suppose, without loss of generality, that $x$ is normalized so that $\max_i\|x_i\|=1$.
We will use a probabilistic argument. Let us consider the random variable $u$ drawn uniformly from $[0,1]$
  and recall that $x^u$ is defined by
$x^u_i=\frac{x_i}{\|x_i\|}$  if $\|x_i\|^2>u$ or $x^u_i = 0$ if $\|x_i\|^2\leq u$.
We will show that
$\frac12\frac{\EE \sum_{ij}w_{ij}\|x_i^u-\rho_{ij}x_j^u\|}{\EE \sum_id_i\|x^u_i\|} \leq  \sqrt{\frac52 \RRReta(x) }$,
which implies that at least one of the realizations of $u$ must satisfy the inequality, and proves the Lemma.

We start by showing that, for each edge $(i,j)$,
\begin{equation}\label{theorem_VGL_Rd_CI_hardbound_f3_SO3}
\EE\|x^u_i-\rho_{ij}x^u_j\| \leq \frac{\sqrt{5}}2 \|x_i-\rho_{ij}x_j\|\left(\|x_i\|+\|x_j\|\right).
\end{equation}
Without loss of generality we can consider $\rho_{ij}=I$ and $\|x_j\|\leq\|x_i\|$ and get,
$$\EE\|x^u_i-x^u_j\| = \|x_j\|^2\left\|\frac{x_i}{\|x_i\|}-\frac{x_j}{\|x_j\|}\right\| + \left( \|x_i\|^2 - \|x_j\|^2 \right).$$
Thus, it suffices to show
\[
\|x_j\|^2\left\|\frac{x_i}{\|x_i\|}-\frac{x_j}{\|x_j\|}\right\| + \left( \|x_i\|^2 - \|x_j\|^2 \right) \leq
\frac{\sqrt{5}}{2}
\|x_i-x_j\|\left(\|x_i\|+\|x_j\|\right),
\]
which is a consequence of Proposition \ref{prop:OdCheegerLinearAlg} for $y = \frac{x_j}{\|x_j\|}$, $z = \frac{x_i}{\|x_i\|}$ and $\alpha = \frac{\|x_i\|}{\|x_j\|}$.
%
Now, using (\ref{theorem_VGL_Rd_CI_hardbound_f3_SO3}), the linearity of expectation, and the Cauchy-Schwartz inequality we have
\begin{eqnarray*}
\EE \sum_{ij}w_{ij}\|x^u_i-\rho_{ij}x^u_j\| 
 & \leq  & \frac{\sqrt{5}}2\sum_{ij} w_{ij}  \|x_i-\rho_{ij}x_j\| (\|x_i\|+\|x_j\|) \\
 & \leq  & \frac{\sqrt{5}}2 \sqrt{\sum_{ij} w_{ij} \|x_i-\rho_{ij}x_j\|^2} \sqrt{\sum_{ij} w_{ij} (\|x_i\|+\|x_j\|)^2}.
\end{eqnarray*}
Since $\sum_{ij} w_{ij} \|x_i-\rho_{ij}x_j\|^2 = 2\RRReta(x)\sum_{i}d_i\|x_i\|^2$ and
\[
\sum_{ij} w_{ij} (\|x_i\|+\|x_j\|)^2 \leq 2 \sum_{ij} w_{ij} (\|x_i\|^2+\|x_j\|^2) = 4 \sum_id_i\|x_i\|^2,
\]
we have
\[
\EE \sum_{ij}w_{ij}\|x^u_i-\rho_{ij}x^u_j\| \leq \frac{\sqrt{5}}2\sqrt { 8 \RRReta(x) }  \sum_id_i\|x_i\|^2 = \frac{\sqrt{5}}2\sqrt { 8 \RRReta(x) }  \EE \sum_i d_i\|x^u_i\|= 2\sqrt { \frac52 \RRReta(x) }  \EE \sum_i d_i\|x^u_i\|,
\]
which completes the proof. 
When $d=1$ the sharper result can be obtained by noting that (\ref{theorem_VGL_Rd_CI_hardbound_f3_SO3}) holds even without the $\frac{\sqrt{5}}2$ factor.
\end{proof}

In the context of full $\SSS^{d-1}$ Synchronization, the intuition given to justify a requirement in connectivity of the graph is that it forces the solution of the relaxed problem to have balanced norm across the vertices of the graph. The following Lemma makes this thought precise.

\begin{lemma}\label{lemma:CI_Od_balanced}
Given $x\in\RR^{dn}$, there exists $\alpha_x\geq0$, such that $r_x = x - \alpha_x\tilde{x}$ satisfies  $\|r_x\|_{D_1}^2 \leq \frac{\RRReta(x)}{\lambda_2(\LLL_0)}\|x\|_{D_1}^2$.
\end{lemma}
\proofb{
Let us define $n_x\in\RR^n$ by $(n_x)_i = \|x_i\|$ and recall $\tilde{x}$ defined in (\ref{def:tildex}). 
We now set $\alpha_x = \argmin_\alpha\|n_x - \alpha {\bf 1}\|_{D_0}$.
A simple calculation reveals that this gives
\[
  \alpha_{x} = \frac{{\bf 1}^{T} D_{0} n_{x}}{{\bf 1}^{T} D_{0} {\bf 1}}.
\]
Since $n_x$ is a non-negative vector, $\alpha_{x}$ is non-negative as well.
Let us also define $u_x\in\RR^n$ so that $(r_x)_i = (u_x)_i \tilde{x}_i$.
This implies that $u_x = n_x - \left(\frac{{\bf 1}^TD_0n_x}{{\bf 1}^TD_0{\bf 1}}\right){\bf 1}$. Thus,
\[
u_x^TL_0u_x = n_x^TL_0n_x = \frac12\sum_{ij}w_{ij}(\|x_i\|-\|x_j\|)^2\leq \frac12\sum_{ij}w_{ij}\|x_i-\rho_{ij}x_j\|^2 = \RRReta(x)\|x\|_{D_1}^2
\]
Since $u_x^TD_0{\bf 1} = 0$, we have $\frac{(u_x)^TL_0u_x}{\|u_x\|_{D_0}^2}\geq \lambda_2(\LLL_0)$. This shows that $\|r_x\|_{D_1}^2 = \|u_x\|_{D_0}^2 \leq \frac1{\lambda_2(\LLL_0)}\RRReta(x)\|x\|_{D_1}^2$.
}

This allows one to bound the volume of the subset of vertices where the norm of $x$ is not typical. Let us first define this set.

\begin{definition}\label{illbalancednormalizationset}
Given $x\in\RR^{dn}$, normalized so that $\|x\|_{D_1}^2 = \vol(G)$, and a positive number $\delta$, we define the Ill-balanced vertex subset of the graph $G$ as $ \IIIb_x(\delta) = \{ i \in V : \left| \|x_i\| - 1  \right|\geq \delta \}$.
\end{definition}

The volume of $\IIIb_x(\delta)$ is controlled by the following Lemma.

\begin{lemma}\label{lemma:OdSynchBoundforIBkTIGHT}
Let $x\in\RR^{dn}$ satisfy $\|x\|_{D_1}^2 = \vol(G)$. Then,
\[
\frac{\vol(\IIIb_x(\delta))}{\vol(G) } \leq \frac{4}{\delta^2} \frac{\RRReta(x)}{\lambda_2(\LLL_0)}.
\]
\end{lemma}

\proofb{
Lemma \ref{lemma:CI_Od_balanced} guarantees the existence of $\alpha_{x}\in\RR^+$ such that $r_{x} = x - \alpha_{x} \tilde{x}$ satisfies $\|r_{x}\|_{D_1}^2 \leq \frac{\RRReta(x)}{\lambda_2(\LLL_0)}\|x\|^2_{D_1}$.

Let us start by bounding $\alpha_x$; by the triangle inequality,
\[
\left(1-\alpha_x\right)^2\vol(G) = \left(\|x\|_{D_1}-\alpha_x\|\tilde{x}\|_{D_1}\right)^2 \leq  \|r_{x}\|_{D_1}^2 \leq \frac{\RRReta(x)}{\lambda_2(\LLL_0)}\vol(G),
\]
which implies $\left(1-\alpha_x\right)^2 \leq \frac{\RRReta(x)}{\lambda_2(\LLL_0)}$.

If $i\in \IIIb_x(\delta)$ then $\left| \|x_i\| - 1 \right|  \geq \delta$, which implies $\|(r_x)_i\| = \left| \|x_i\| - \alpha_x \right| \geq \left| \|x_i\| - 1 \right| - \left|1 - \alpha_x \right|  \geq   \delta - \sqrt{\frac{\RRReta(x)}{\lambda_2(\LLL_0)}}$.
Squaring both sides of the inequality and summing over all $i\in \IIIb_x(\delta)$ gives,
\begin{equation}\label{form:OdCheeger:inLemma7delta0}
\frac{\RRReta(x)}{\lambda_2(\LLL_0)}\vol(G) \geq \|r_x\|_{D_1}^2 \geq \sum_{i\in \III b_k} d_i \|(r_x)_i\|_2^2 \geq \vol(\III b_k) \left( \delta - \sqrt{\frac{\RRReta(x)}{\lambda_2(\LLL_0)}} \right)^2,
\end{equation}
as long as $\delta > \sqrt{\frac{\RRReta(x)}{\lambda_2(\LLL_0)}}$.
Let us separate in two cases:

If $\frac{\delta}2 > \sqrt{\frac{\RRReta(x)}{\lambda_2(\LLL_0)}}$, then, using (\ref{form:OdCheeger:inLemma7delta0}) we have,
\[
 \frac{\vol(\III b_k)}{\vol(G) } \leq \frac{\RRReta(x)}{\lambda_2(\LLL_0)}\left( \delta - \sqrt{\frac{\RRReta(x)}{\lambda_2(\LLL_0)}} \right)^{-2} \leq \frac{\RRReta(x)}{\lambda_2(\LLL_0)}\left( \delta - \frac{\delta}2 \right)^{-2} = \frac4{\delta^2}\frac{\RRReta(x)}{\lambda_2(\LLL_0)}.
\]

If, on the other hand, $\frac{\delta}2 \leq \sqrt{\frac{\RRReta(x)}{\lambda_2(\LLL_0)}}$, then, since $\frac{\vol(\III b_k)}{\vol(G) }\leq 1$,
\[
 \frac{\vol(\III b_k)}{\vol(G) } \leq 1 \leq \frac{\RRReta(x)}{\lambda_2(\LLL_0)}\left( \frac{\delta}2 \right)^{-2} = \frac4{\delta^2}\frac{\RRReta(x)}{\lambda_2(\LLL_0)}.
\]
}

By placing an upper bound on the number of Ill-balanced vertices, (Lemma \ref{lemma:OdSynchBoundforIBkTIGHT}) we can control how much $\RRReta(x)$ is affected when we locally normalize $x$. This is achieved in the following Lemma, which contains the central technical result regarding full $\SSS^{d-1}$ Synchronization.
\begin{lemma}\label{lemma:CI_Sd_CIfull}
 For every $x\in\RR^{dn}$,  $\RRReta(\tilde{x}) \leq \frac{44}{\lambda_2(\LLL_0)}\RRReta(x)$.
\end{lemma}

\proofb{We want to bound $\RRReta(\tilde{x}) = \frac1{2\vol(G)} \sum_{ij}w_{ij}\|\tilde{x}_i - \rho_{ij}\tilde{x}_j\|^2$. Without loss of generality we can assume $\|x\|_{D_1}^2 = \vol(G)$. Let $0<\gamma<1$, then
\begin{eqnarray*}
\RRReta(\tilde{x}) & \leq & \frac1{2\vol(G)} \left( \sum_{i\in\IIIb_x(\gamma)}\sum_jw_{ij}\|\tilde{x}_i - \rho_{ij}\tilde{x}_j\|^2 + \sum_{j\in\IIIb_x(\gamma)}\sum_iw_{ij}\|\tilde{x}_i - \rho_{ij}\tilde{x}_j\|^2 + \sum_{i,j\notin\IIIb_x(\gamma)}w_{ij}\|\tilde{x}_i - \rho_{ij}\tilde{x}_j\|^2\right) \\
& \leq & 4\frac{\vol(\IIIb_x(\gamma))}{\vol(G)} + \frac1{2\vol(G)}  \sum_{i,j\notin\IIIb_x(\gamma)}w_{ij}\|\tilde{x}_i - \rho_{ij}\tilde{x}_j\|^2.
\end{eqnarray*}

By Lemma \ref{lemma:OdSynchBoundforIBkTIGHT} we have $4\frac{\vol(\IIIb_x(\gamma))}{\vol(G)}\leq \frac{16}{\gamma^2}\frac{\RRReta(x)}{\lambda_2(\LLL_0)}$.

Note that, for any $y,z\in\RR^{d}$, $\left\|\frac{y}{\|y\|}-\frac{z}{\|z\|} \right\| \leq \frac{\|y-z\|}{\min\{\|y\|,\|z\|\}}$. By setting $y=x_i$ and $z=\rho_{ij}x_j$ we get $\|\tilde{x}_i - \rho_{ij}\tilde{x}_j\|\leq \frac{\left\|x_i - \rho_{ij}x_j\right\|}{\min\{\|x_i\|,\|x_j\|\}}$. This implies that
\begin{eqnarray*}
\frac1{2\vol(G)}  \sum_{i,j\notin\IIIb_x(\gamma)}w_{ij}\|\tilde{x}_i - \rho_{ij}\tilde{x}_j\|^2 &\leq& \frac1{2\vol(G)}  \sum_{i,j\notin\IIIb_x(\gamma)}w_{ij}\left(\frac{\left\|x_i - \rho_{ij}x_j\right\|}{\min\{\|x_i\|,\|x_j\|\}}\right)^2\\
&\leq& \frac1{2\vol(G)}\frac1{(1-\gamma)^2}  \sum_{i,j\notin\IIIb_x(\gamma)}w_{ij}\left\|x_i - \rho_{ij}x_j\right\|^2.
\end{eqnarray*}

This means that
\begin{eqnarray*}
 \RRReta(\tilde{x}) &\leq & \frac{16}{\gamma^2}\frac{\RRReta(x)}{\lambda_2(\LLL_0)} + \frac1{2\vol(G)}\frac1{(1-\gamma)^2}  \sum_{i,j\notin\IIIb_x(\gamma)}w_{ij}\left\|x_i - \rho_{ij}x_j\right\|^2 \\
 &\leq & \left( \frac{16}{\gamma^2}\frac{1}{\lambda_2(\LLL_0)} + \frac1{(1-\gamma)^2} \right)\RRReta(x).
\end{eqnarray*}

Since $\lambda_2(\LLL_0)\leq 1$ (see, e.g., \cite{Chung_ComplexGraphs}), it is possible to pick $\gamma$ (e.g. $0.7$) such that $\frac{16}{\gamma^2}\frac{1}{\lambda_2(\LLL_0)} + \frac1{(1-\gamma)^2} \leq \frac{44}{\lambda_2(\LLL_0)}$.
}

\subsection{Proofs for Synchronization in $O(d)$}

As described, the rounding procedure for $O(d)$ Synchronization is based on the polar decomposition. We need to understand how much the polar decomposition can potentially affect the penalty on each edge. With this purpose we use the following result of Li~\cite{Li_StabilityPolarDecomposition}.
\begin{lemma}[Theorem 1 in \cite{Li_StabilityPolarDecomposition}]\label{lemma:PolarDecompositionStabilityLI}
Let $A,B\in\CC^{d\times d}$ be non-singular matrices with polar decompositions $A = U(A)P$ and $B = U(B)P'$. Then $\|U(A) - U(B)\|_F \leq \frac{2}{\sigma_{\min}(A) + \sigma_{\min}(B)}\|A-B\|_F$, where $\sigma_{\min}(A)$ is the smallest singular value of the matrix $A$.
\end{lemma}

This means that, in order to bound possible instabilities of the polar decomposition, one needs to control the size of the smallest singular value of $X_i=\left[x_i^1\cdots x_i^d\right]$. To achieve this, we will introduce a notion similar to  $\IIIb_x(\delta)$ but designed to take into account local orthogonality instead of local normalization.

\begin{definition}\label{definition:IIIbxy}
Given $x,y\in\RR^{dn}$ two $D_1$-orthogonal vectors, normalized so that $\|x\|_{D_1}^2 = \|y\|_{D_1}^2 = \vol(G)$, and a positive number $\delta$, we define the following Ill-balanced vertex subset of the graph $G$ as
\[
\III b_{xy}(\delta)  = \left\{i\in V: |\langle x_i, y_i\rangle| \geq \delta \right\}.
\]
\end{definition}

The next Lemma shows that, for an edge whose incident vertices are not ill-balanced (see Definitions \ref{definition:IIIbxy} and \ref{illbalancednormalizationset}), the polar decomposition only slightly affects the penalty function.

\begin{lemma}\label{lemma:CI_stabilityPDonBBB}
Let $x^1,\dots,x^d\in\RR^{dn}$ be $D_1$-orthogonal vectors, normalized so that $\|x^k\|^2_{D_1}=\vol(G)$. Let us define the ``balanced'' set $\BBB$ as
 the complement of $\bigcup_{k\in [d]}\left(\IIIb_{x^k}\left(\frac1{8d}\right) \cup \bigcup_{m\in[d]\setminus \{k\}} \IIIb_{x^kx^m}\left(\frac1{2d}\right) \right)$. For all $i,j\in \BBB$, we have \[ \|U(X_i) - \rho_{ij}U(X_j)\|_F \leq \sqrt{2} \|X_i - \rho_{ij}X_j\|_F.\]
\end{lemma}

\proofb{
For $i\in\BBB$, consider the gram matrix $X_i^TX_i$. Its $k$-th diagonal entry satisfies $\|x_i^k\|^2\geq \left(1-\frac1{8d}\right)^2 \geq 1 - \frac1{4d}$. On the other hand the non-diagonal entries are, in magnitude, smaller or equal to $\frac1{2d}$. By the Gershgorin circle theorem, the smallest eigenvalue of $X_i^TX_i$, which is equal to $\sigma_{\min}(X_i)^2$, satisfies $\sigma_{\min}(X_i)^2\geq 1 - \frac1{4d}-(d-1)\frac1{2d}$. Hence, $\sigma_{\min}(X_i)\geq \frac{1}{\sqrt2}$. By observing that $U(\rho_{ij}X_j)=\rho_{ij}U(X_j)$, and using Lemma \ref{lemma:PolarDecompositionStabilityLI}, we get $\|U(X_i) - \rho_{ij}U(X_j)\|_F \leq \sqrt{2} \|X_i - \rho_{ij}X_j\|_F$.
}

The last step is to control the size of the ill-balanced sets.

\begin{lemma}\label{lemma:CI_sizeofIIIbxydelta}
Let $x,y\in\RR^{dn}$ be $D_1$-orthogonal vectors such that $\|x\|_{D_1}^2 = \|y\|_{D_1}^2 = \vol(G)$. Then,
\[
\frac{\vol\left(\III b_{xy}\left(\frac1{2d}\right)\setminus \left( \III b_{x}\left(\frac1{8d}\right) \cup \III b_{y}\left(\frac1{8d}\right) \right)\right)}{\vol(G)}\leq 4(8d)^2\frac{\RRReta(x) + \RRReta(y)}{\lambda_2(\LLL_0)}.
\]
\end{lemma}

\proofb{
Let us consider the vector $u = \frac1{\sqrt2}(x + y)$. It satisfies $\|u\|_{D_1}^2 = \vol(G)$ and by the triangle inequality on the norm $\|\cdot\|_{L_1}$, $\RRReta(u) \leq \RRReta(x) + \RRReta(y)$. By Lemma \ref{lemma:OdSynchBoundforIBkTIGHT} we get $\frac{\vol\left(\IIIb_u\left(\frac1{8d}\right)\right)}{\vol(G)}\leq 4(8d)^2\frac{\RRReta(x) + \RRReta(y)}{\lambda_2(\LLL_0)}$. We conclude the proof by noting that $\III b_{xy}\left(\frac1{2d}\right) \subset  \III b_{x}\left(\frac1{8d}\right) \cup \III b_{y}\left(\frac1{8d}\right) \cup \IIIb_u\left(\frac1{8d}\right)$. In fact, if $i \not \in \III b_{x}\left(\frac1{8d}\right) \cup \III b_{y}\left(\frac1{8d}\right) \cup \IIIb_u\left(\frac1{8d}\right)$ then $|\langle x_i,y_i\rangle| 
=\left| \|u_i\|^2 - \frac{\|x_i\|^2  + \|y_i\|^2}2\right| \leq \left(1+\frac1{8d}\right)^2 -  \left(1-\frac1{8d}\right)^2 = \frac1{2d}$.
}

At this point we have build all the foundations needed for the proof of the central lemma regarding $O(d)$ Synchronization.

\begin{lemma}\label{lemma:CI_Od_MAINLEMMA}
 Given $x^1,\dots,x^d\in\RR^{dn}$ such that $\langle x^k,x^l\rangle_{D_1}=0$ for all $k\neq l$, consider the potential $g:V\to O(d)$ given as $g_i = U\left(X_i\right)$
where $X_i = \left[x_i^1\cdots x_i^d\right]$ and $U(X)$ is the closest (in the Frobenius norm) orthogonal matrix of $X$. If $X_i$ is singular $U(X_i)$ is simply set to be $\Id_d$. Then,
\[
 \fruvtwo(g) \leq \left( 2d^{-1} + 2^{10} d^3 \right) \frac1{\lambda_2(\LLL_0)}\sum_{i=1}^d\RRReta\left(x^i\right).
\]
\end{lemma}

\proofb{

Let us consider $\BBB$ as defined in Lemma \ref{lemma:CI_stabilityPDonBBB}, meaning $\BBB^c = \bigcup_{k\in [d]}\left(\IIIb_{x^k}\left(\frac1{8d}\right) \cup \bigcup_{m\in[d]\setminus \{k\}} \IIIb_{x^kx^m}\left(\frac1{2d}\right) \right)$. We want to bound $\fruvtwo(g) = \frac1{2d\vol(G)}\sum_{ij}w_{ij}\|g_i - \rho_{ij}g_j\|_F^2$. Since $\left(\BBB\times\BBB\right)^c \subset \left(\BBB^c\times V \right)\cup \left(V\times \BBB^c \right)$,
\begin{eqnarray*}
 \sum_{ij}w_{ij}\|g_i - \rho_{ij}g_j\|_F^2 & \leq & \sum_{(i,j)\in \BBB\times \BBB}w_{ij}\|U(X_i) - \rho_{ij}U(X_j)\|_F^2 + 2 \sum_{i\in\BBB^c}\sum_{j\in V}w_{ij}\|g_i - \rho_{ij}g_j\|_F^2\\
& \leq & 2\sum_{ij}w_{ij}\|X_i - \rho_{ij}X_j\|_F^2 + 8d \vol\left(\BBB^c\right),
\end{eqnarray*}
where the second inequality was obtained by using Lemma \ref{lemma:CI_stabilityPDonBBB} and noting that $O_i$ is an orthogonal matrix. To bound $\vol\left(\BBB^c\right)$ we make use of Lemmas \ref{lemma:OdSynchBoundforIBkTIGHT} and \ref{lemma:CI_sizeofIIIbxydelta} and get that $\frac{\vol\left(\BBB^c\right)}{\vol(G)}$ is bounded above by
\[
\sum_{k\in[d]}\frac{\vol\left(\IIIb_{x^k}\left(\frac1{8d}\right)\right)}{\vol(G)} + \frac12\sum_{k\in[d]}\sum_{m\in[d]\setminus k}\frac{\vol\left(\IIIb_{x^kx^m}\left(\frac1{2d}\right)\setminus \left( \III b_{x}\left(\frac1{8d}\right) \cup \III b_{y}\left(\frac1{8d}\right) \right)\right)}{\vol(G)}
\leq 2^8d^3\frac{\sum_{k\in [d]}\RRReta(x^k)}{\lambda_2(\LLL_0)}.
\]
Since $\sum_{ij}w_{ij}\|X_i - \rho_{ij}X_j\|_F^2 = 2\vol(G)\sum_{k=1}^d\RRReta(x^k)$, we get $\fruvtwo(g) \leq \left( 2d^{-1} + 2^{10}d^3 \right)\frac1{\lambda_2(\LLL_0)}\sum_{k=1}^d\RRReta(x^k)$.
}

%
%
%

\section{An unsquared version of the frustration constant}

We formulated the Synchronization problem as minimizing the square of the Frobenius norm of the incompatibilities (in some sense, an $\ell_2$ penalty function). This penalty function is particularly nice because it is close in spirit to the Rayleigh quotient formulation of an eigenvector problem and thus more related to what the spectral method will try to minimize, as we indeed showed in the theorems above.

On the other hand, considering the sum of the Frobenius norms of the incompatibilities (in some sense, an $\ell_1$ penalty function) will induce sparsity on the edge inconsistencies, meaning that it will favor candidate solutions for which some edges are perfectly correct even if there are some edges with large errors. This type of penalty function is often favorable under some noise models. In fact, the noise model analyzed in \cite{ASinger_2011_angsync} consists of a few randomly chosen edges having a measurement that is randomly drawn with respect to the uniform distribution in the space of possible measurements (in our case $O(d)$, in  \cite{ASinger_2011_angsync} $SO(2)$ ), therefore the original rotation potential will perfectly agree with some edges and have a large error on others. This motivates us to look at $\ell_1$ versions of frustration constants. Let us define
\[
\fruvtwo_1(O) = \frac1{\sqrt{d}\vol(G)}\sum_{ij}w_{ij}\|O_i - \rho_{ij}O_j\|_F,
\]
and the $O(d)$ frustration $\ell_1$ constant of $G$ as $\fruvtwo_{G,1} = \min_{O:V\to O(d)}\fruvone(O)$.
Similarly,
\[
\fruztwo_1(v) = \frac{\sum_{ij}w_{ij}\|v_i-\rho_{ij}v_j\|}{\sum_id_i\|v_i\|}
\]
and the $\ell_1$ constant of $G$, as $\fruztwo_{G,1} = \min_{v:V\to\SSS^{d-1}}\fruzone(v)$. We also define a partial version of it $\fruztwo^\ast_{G,1} = \min_{v:V\to\SSS^{d-1}\cup\{0\}}\fruzone(v)$.

From our results in the above section it is easy to obtain the following Cheeger type inequalities for these frustration constants

\begin{theorem}
Let $\lambda_i(\LLL_1)$ and $\lambda_i(\LLL_0)$ denote the $i$-th smallest eigenvalue of, respectively, the normalized Connection Laplacian $\LLL_1$ and the normalized graph Laplacian $\LLL_0$. Let $\fruztwo^\ast_{G,1}$, $\fruztwo_{G,1}$,and $\fruvtwo_{G,1}$, denote the $\ell_1$ frustration constants defined above.
Then,
\begin{equation}\label{l1thmineq1}
\lambda_1(\LLL_1) \leq \fruztwo^\ast_{G,1} \leq \sqrt{ 10 \lambda_1(\LLL_1)},
\end{equation}
\begin{equation}\label{l1thmineq2}
\lambda_1(\LLL_1) \leq \fruztwo_{G,1} \leq 2\sqrt{\frac{22}{\lambda_2(\LLL_0)} \lambda_1(\LLL_1)},
\end{equation}
and,
\begin{equation}\label{l1thmineq3}
\frac{1}{d} \sum_{i=1}^d\lambda_i(\LLL_1) \leq \fruvtwo_{G,1} \leq 6d \sqrt{\frac{57d}{\lambda_2(\LLL_0)} \sum_{i=1}^d\lambda_i(\LLL_1)}.
\end{equation}
\end{theorem}

\begin{proof}

For any $v:V\to S^{d-1} \cup \{0\}$ we have
\[
\fruztwo(v) = \frac12\frac{\sum_{ij}w_{ij}\|v_i-\rho_{ij}v_j\|^2}{\sum_id_i\|v_i\|^2} \leq \frac{\left(\sum_{ij}w_{ij}\|v_i-\rho_{ij}v_j\|\right)\frac{\max_{ij}\|v_i-\rho_{ij}v_j\|}2}{\left(\sum_id_i\|v_i\|\right)\min\|v_i\|} \leq \frac{\sum_{ij}w_{ij}\|v_i-\rho_{ij}v_j\|}{\sum_id_i\|v_i\|} = \fruztwo_1(v),
\]
which, together with Theorems \ref{theorem:CI_Sd_CIpartial} and \ref{theorem:CI_Sd_CIfull}, gives the lower bound on both (\ref{l1thmineq1}) and (\ref{l1thmineq2}).

Note that Lemma \ref{lemmaalsogoodforl1} actually guarantees that there exists $v:V\to S^{d-1} \cup \{0\}$
such that $\fruztwo_1(v)\leq \sqrt{ 10\lambda_i(\LLL_1) } $ which concludes the proof of (\ref{l1thmineq1}).

Let $v:V\to S^{d-1} \cup \{0\}$ be a solution that satisfies $\fruztwo(v) \leq 44\frac{\lambda_1(\LLL_1)}{\lambda_2(\LLL_0)}$, guaranteed to exist by Theorem \ref{theorem:CI_Sd_CIfull}. We then have,
\begin{eqnarray}
\fruztwo_1(v) & = & \frac{\sum_{ij}w_{ij}\|v_i-\rho_{ij}v_j\|}{\sum_id_i\|v_i\|}   =  \frac1{\vol(G)}\sum_{ij}\left(w_{ij}^{\frac12}\|v_i-\rho_{ij}v_j\|\right)w_{ij}^{\frac12} \nonumber \\
& \leq & \frac1{\vol(G)}\left(\sum_{ij}w_{ij}\|v_i-\rho_{ij}v_j\|^2\right)^{\frac12}\left(\sum_{ij}w_{ij}\right)^{\frac12}
 =  \sqrt{\frac{\sum_{ij}w_{ij}\|v_i-\rho_{ij}v_j\|^2}{\sum_id_i\|v_i\|^2}}
 =  \sqrt{2\fruztwo(v)}, \label{usingCShere}
\end{eqnarray}
where the inequality is obtained using Cauchy-Schwarz. This completes the proof of (\ref{l1thmineq2}).

Inequality (\ref{l1thmineq3}) is shown in the same way as (\ref{l1thmineq2}): since $g_i,g_j\in O(d)$, we have $\|g_i-\rho_{ij}g_j\|_F \leq 2\sqrt{d}$ which gives $\fruvtwo_1(g) \geq \fruvtwo(g)$. On the other hand, using Cauchy Schwarz in the same way as in (\ref{usingCShere}) gives $\fruvtwo_1(g) \leq \sqrt{2\fruvtwo(g)}$, which implies (\ref{l1thmineq3}) and concludes the proof of the Theorem.
%
%
%

\end{proof}

In fact Wang and Singer \cite{Wang_RobustSynchronization} recently showed that under the random outlier's noise model (described above), a semidefinite relaxation of the $O(d)$ Synchronization problem formulated with the $\ell_1$ penalty function was able to recover the ground truth solution with high probability, provided the underlying graph is drawn from the Erd\H{o}s-R\'enyi random graph model and the ratio of outliers is below a certain threshold.

\section{Tightness of results}

Let us consider the ring graph on $n$ vertices $G_n=(V_n,E_n)$ with $V_n = [n]$ and $E=\{(i,(i+1)\hspace{-0.25cm}\mod n),\, i\in[n]\}$ with the edge weights all equal to $1$ and $\rho:V\to O(d)$ as $\rho_{(n,1)} = -\Id$ and $\rho = \Id$ for all other edges.
Define $x\in\RR^{dn}$ by $x_k = \left[2\frac{k}{n}-1,0,\dots,0\right]^T.$ It is easy to check that $\RRReta(x) = \OOO(n^{-2})$ and that,
for any $u> 0$, if $x^u\not\equiv 0$, there will have to be at least one edge that is not compatible with $x^u$, implying $\fruztwo(x^u) \geq \frac1{2n}$. This shows that the $1/2$ exponent in Lemma \ref{lemma:CI_Sd_fromTrevisan_PARTIAL} is needed. In fact, by adding a few more edges to the graph $G_n$ one can also show the tightness of Theorem \ref{theorem:CI_Sd_CIpartial}: Consider the ``rainbow'' graph $H_n$ that is constructed by adding to $G_n$, for each non-negative integer $k$ smaller than $n/2$, an edge between vertex $k$ and vertex $n-k$ with $\rho_{(k,n-k)} = -\Id$. The vector $x$ still satisfies $\RRReta(x) = \OOO(n^{-2})$, however, for any non-zero vector $v:V\to\SSS^{d-1}\cup\{0\}$, it is not hard to show that $\fruztwo(v)$ has to be of order at least $n^{-1}$, meaning that $\fruztwo_G^\ast$ is $\Omega(\sqrt{\lambda_1(\LLL_1)})$. This also means that, even if considering $\fruztwo_G^\ast$, one could not get a linear bound (as provided by Lemma \ref{lemma:CI_Sd_CIfull}) without the control on $\lambda_2(\LLL_0)$.




Theorem \ref{theorem:CI_Od_MAIN} provides a non-trivial bound only if $\lambda_2(\LLL_0)$ is sufficiently large. It is clear that if one wants to bound full frustration constants, a dependency on $\lambda_2(\LLL_0)$ is needed. It is, nevertheless, non-obvious that this dependency is still needed if we consider partial versions of $O(d)$ frustration constants, $\fruvonepar_G$ or $\fruvtwopar_G$. This can, however, be illustrated by a simple example in $O(2)$; consider a disconnected graph $G$ with two sufficiently large complete components, $G^1=(V^1,E^1)$ and $G^2=(V^2,E^2)$. For each edge let $\rho_{i,j} = \tiny{ \left[ \begin{array}{cc}   -1 & 0 \\  0 & 1   \end{array}  \right] }$.
It is clear that the vectors $x^1$ and $x^2$ defined such that $x^1_i = [0,1_{V^1}(i)]^T$ and $x^2_i = [0,1_{V^2}(i)]^T$ are orthogonal to each other and lie in the null space of the graph Connection Laplacian of $G$. This implies that $\lambda_2(\LLL_1)=0$. On the other hand, it is straightforward to check that $\fruvtwo_G^\ast$ is not zero because it is impossible to perfectly synchronize the graph (or any of the components, for that matter).

\section{Concluding Remarks}

Synchronization is a challenging problem. Recent discoveries suggest that spectral relaxations are promising as feasible methods to solve this problem. In fact, in \cite{ASinger_2011_angsync},  probability guarantees of performance are given for the performance of a spectral method to solve the $SO(2)$ synchronization problem under a certain random noise model. Nevertheless, to the best of our knowledge, Algorithm \ref{Alg:OdSync} is the first method for $O(d)$ synchronization having a (deterministic) worst case performance guarantee. As one would expect, the worst case performance is significantly weaker than the kind of probabilistic guarantees, given a specific noise model, e.g. as the one given in \cite{ASinger_2011_angsync}. In fact, the guarantees in \cite{ASinger_2011_angsync} and \cite{Wang_RobustSynchronization} are given in terms of distance between the candidate solution and the ground truth, while the one we provide is given in terms of the compatibility error. Recently, in the context of the Phase Retrieval problem, the guarantees on this paper were used to obtain guarantees on the distance between the candidate solution and the ground truth (see \cite{Alexeev_etal_Phaseless}).

In special applications one knows, a priori, that every element in the potential has positive determinant. This corresponds to a synchronization problem in $SO(d)$.
 Although this can be viewed as a special case of the $O(d)$ problem it is expected that the additional structure can be leveraged to improve the algorithm (and the analysis). In particular, the first $d-1$ columns of a matrix in $SO(d)$ completely determine the matrix. This suggests that the $SO(d)$ synchronization problem is solvable by just the first $d-1$ eigenvectors of the Connection Laplacian, instead of the $d$ first ones. In fact, the $SO(2)$ Synchronization problem is equivalent to the $\SSS^1$ localization one, and the guarantees for the $\SSS^1$ localization problem were given solely in terms of the first eigenvalue of the graph Connection Laplacian. We leave the improved $SO(d)$ analysis for future work.

The performance guarantee for this algorithm relies on the fact that the vectors obtained are $D_1$-orthogonal. However, in practice, due to possible errors in the calculations this condition might be perturbed and the $D_1$ inner products of these vectors, although small, may no longer be exactly zero. It is easy to adapt the analysis to this setting and show that it is in fact robust to such perturbations.

The results in this paper also suggest an alternative to Algorithm \ref{Alg:OdSync} which corresponds to, instead of solving the eigenvector problem, determining the $d$ vectors sequentially and, at each step, constraining on the vector being locally orthogonal to the previous ones (this can still be done efficiently, see \cite{Golub_ConstRayleigh}). After the $d$ vectors are obtained one can simply locally normalize each one and output that as the Synchronization solution candidate. The issue with this method is that its iterative nature~\footnote{The fact that the calculation of one of the vectors greatly depends on the ones already computed.} makes its analysis more difficult, as it is hard to guarantee that small errors in the first few vectors would not greatly affect the remaining ones. Also, numerical simulations suggest that the performance, in practice, of both methods is roughly the same. Independently of which method is used, the solution, although guaranteed to have a certain performance, is not guaranteed to be a local optimum. Recently, Boumal et al.~\cite{BoumalManifoldSynch} suggest that an iterative smooth optimization (in manifolds) method, when started in the solution given by the rounding procedure, can be used to locally search for a better solution.

One might argue that, in some applications, the weights on the edges of the graph do not have a clear meaning. The reason being that we may be given a few relative measurements $\rho_{ij}$ and it is unclear how to give weights to such measurements. In such cases, since we have the freedom of choosing the weights of the edges and Theorems \ref{theorem:CI_Sd_CIfull} and \ref{theorem:CI_Od_MAIN} suggest that our method will work better with a large $\lambda_2(\LLL_0)$, one could compute the weights of the edges in such a way that $\lambda_2(\LLL_0)$ is maximized. This problem is solved in~\cite{JSun_etall_2006}. The caveat is that, the new weights will affect the way the compatibility error is measured, as well as the eigenvalues of the Connection Laplacian. It is thus still unclear if such an approach would improve the method. Another interesting possible outcome of a procedure of this nature is a possible ranking of the edges, large weights would likely tend to be given to edges that are more important to ensure the connectivity of the graph.

The classical Cheeger inequality has an analogous result on smooth manifolds (actually, the first to be shown~\cite{JCheeger_1970}). One interesting question is whether the theorems in this paper have an analogous smooth version. One difficulty is to understand what would correspond to the frustration constant in the smooth case. Recent work on Vector Diffusion Maps~\cite{ASinger_HTWu_2011_VDM} suggests that an analogous result on smooth manifolds would be related to the parallel transport and its incompatibility and some results in Differential Geometry \cite{WBallmann_JBruning_GCarron_2002} suggest that the Holonomy could be a geometric property that corresponds to the frustration constant. These suggestions are ``coherent'' because Holonomy can, in some sense, be viewed as the incompatibility of the Parallel transport (due to the curvature of the manifold).

In some cases the incompatibilities have some structure; the continuous setting described above may be such an example. Although, in this paper, we make no attempt to understand or take advantage of such structure, we believe this would be an interesting direction for future work and we direct the reader to papers on which topological tools are used to understand (and leverage) the structure of inconsistencies in synchronization-like problems \cite{Candogan_GamesSynch,Jiang_stastisticalranking,Molavi_TopologicalSynch}.

\subsection*{Acknowledgments}
The authors thank Hau-Tieng Wu for interesting discussions on the topic of this paper and Leor Klainerman for reading a preliminary version of this manuscript. The authors also acknowledge the comments made by the referees and the editor, that helped to significantly improve this paper. Afonso S. Bandeira was supported by Award Number DMS-0914892 from the NSF. Amit Singer was partially supported by Award Numbers FA9550-09-1-0551 and FA9550-12-1-0317 from AFOSR, Award Number R01GM090200 from the National Institute of General Medical Sciences, the Alfred P. Sloan Foundation, and Award Number LTR DTD 06-05-2012 from the Simons Foundation. Daniel A. Spielman was supported by the NSF under Grant No. 091548.


\bibliographystyle{plain}
\bibliography{OdCheeger}

\appendix

\section{Some Technical Steps}

\begin{proposition}\label{prop:OdCheegerLinearAlg}
For any $y$ and $z$ unit vectors in $\RR^{d}$, the following holds for any $\alpha\geq1$,
\[
 \|y-z\| + \alpha^2 - 1 \leq \frac{\sqrt{5}}2 \|y-\alpha z \|(1+\alpha).
\]
\end{proposition}

\proofb{Let $t = \|y-z\|$, which implies that $0\leq t\leq 2$. Since $y$ and $z$ are unit vectors it is straightforward to check that $ \|y-\alpha z \| = \sqrt{1 + \alpha^2 - 2\alpha\left( 1 - \frac12t^2 \right)}$.
Thus, it suffices to show
\begin{equation}\label{form:onAppendixofOdCheeger}
 t + \alpha^2 - 1 \leq \frac{\sqrt{5}}2 \sqrt{1 + \alpha^2 - 2\alpha\left( 1 - \frac12t^2 \right)}(1+\alpha),
\end{equation}
for all $0\leq t\leq 2$ and $\alpha\geq1$. Since both sides of (\ref{form:onAppendixofOdCheeger}) are positive, it is enough to show the inequality with both sides squared. Squaring and rearranging yields,
\[
 t^2 +  2t\alpha^2 - 2t  + \left(\alpha^2-1\right)^2 \leq \frac54 \left(   (\alpha^2-1)^2  + \alpha t^2(1+\alpha)^2\right).
\]
This is equivalent to the non-negativity, in the interval $[0,2]$, of a certain quadratic function of $t$:
\[
 \left(\frac54\alpha(1+\alpha)^2 - 1 \right)t^2 - \left( 2\alpha^2 - 2 \right)t + \frac14\left(\alpha^2-1\right)^2 \geq 0.
\]
Since, for $\alpha\geq1$, $$\left( 2\alpha^2 - 2 \right)^2 - 4\left(\frac54\alpha(1+\alpha)^2 - 1 \right)\frac14\left(\alpha^2-1\right)^2 = \left(\alpha^2-1\right)^2\left( 4 - \frac54\alpha(1+\alpha)^2 + 1  \right)\leq 0,$$ the quadratic is always non-negative and thus non-negative in $[0,2]$.
}

\end{document}